\documentclass[12pt]{article}
\usepackage{amssymb,amsthm}
\usepackage{graphicx}
\usepackage{amsmath}
\usepackage{mathrsfs}
\usepackage{bm}
\usepackage[show]{ed}
\newtheorem{proposition}{Proposition}[section]
\newtheorem{corollary}[proposition]{Corollary}

\newtheorem{theorem}[proposition]{Theorem}

\numberwithin{equation}{section}

\title{Time-Dependent Fluid-Structure Interaction}

\date{\today}
\author{{\sc George C. Hsiao}\thanks{Department of Mathematical Sciences, University of Delaware,
Newark, DE 19716-2553, USA, \quad
Email: {\tt hsiao@math.udel.edu}} \quad
{\sc  Francisco-Javier Sayas} \thanks{Department of Mathematical Sciences, University of Delaware,
Newark, DE 19716-2553, USA,\quad
 Email: {\tt fjsayas@math.udel.edu}}\quad
{\sc  Richard J. Weinacht} \thanks{Department of Mathematical Sciences, University of Delaware,
Newark, DE 19716-2553, USA, \quad
Email: {\tt weinacht@math.udel.edu}.}}

\begin{document}
\maketitle

\begin{abstract}
\noindent 

The problem of determining the manner in which an incoming acoustic wave is scattered by an elastic body immersed in a fluid is one of central importance in detecting and identifying submerged objects.  The problem is generally referred  to as a fluid-structure interaction and is  mathematically formulated as a time-dependent transmission problem.  In  this paper,  we 
consider a typical fluid-structure interaction problem by  using a coupling procedure which  
reduces  the problem to a nonlocal initial-boundary problem in the elastic  body with a system of  integral equations on the interface between the  domains occupied  by the elastic body and the fluid. We analyze this nonlocal problem by the Lubich  approach via the Laplace transform,  an essential  feature of which is that it works directly on data in the time domain rather than in the transformed domain.  Our results may serve as a mathematical foundation for  treating  time-dependent 
fluid-structure interaction problems by  convolution quadrature  coupling  of FEM
and  BEM. 
\end{abstract}

{\bf key words}: Fluid-structure interaction,  Coupling procedure,  Kirchhoff representation formula, 
Retarded potential, Laplace transform,  Boundary integral equation,  Variational formulation, Sobolev space.

\noindent
{\bf Mathematics Subject Classifications(1991)}: 35J20, 35L05, 45P05, 65N30.65N38

\section{Introduction} 
The problem of determining the manner in which an incoming acoustic wave is scattered by an elastic body immersed in a fluid is one of central importance in detecting and identifying submerged objects.  The problem is generally referred  to  as a fluid-structure interaction and is  mathematically formulated as an initial-boundary  transmission problem.   However,  most  of the investigations  study typical fluid-structure interaction problems confined to  the
 time-harmonic setting;  various  numerical methods, sometimes competitive, sometimes complementary,  have been developed.  In this regard, the governing system of partial differential equations is usually replaced by integral equations and it is these formulations upon which most numerical approximations are based. The acoustic equation is replaced by a boundary integral equation while the elastic body is treated in various ways; sometimes using an integral equation, either a  boundary or domain equation, or  alternatively using a weak or variational formulation leading to finite element approximations  
(see e.g., \cite{HJ:1985, HKS:1988, EvHe:1990, BiMa:1991, LuMa:1995,  Sc:2008},  to name a few). 

In  this paper,  we study and analyze a typical fluid-structure interaction problem in the time domain. Motivated by the time-harmonic fluid-structure interaction problems, we apply  a coupling procedure which is a combination of  a field equation and a boundary  integral equation. The essence of the procedure is to reduce the problem to a nonlocal initial- boundary problem in the elastic  body with integral equations on the interface between the  domains occupied  by the elastic body and the fluid.  However, in contrast to the time-harmonic setting,  the integral equations which are derived from the Kirchhoff formula are not only nonlocal in space but also nonlocal in time. This makes the analysis complicated,  in particular with respect to the choice of  appropriate solution function spaces.  To circumvent   this difficulty, we analyze this nonlocal  initial-boundary problem by the Lubich  approach via the Laplace transform as in  \cite{LaSa:2009b, Sa:2012}. The Lubich approach has been employed in the development of  numerical approximations  for  some fluid-structure interaction problems in the engineering literature (see, e.g., \cite{EsAn:1991, Sch:2001, PeBe:2009}),  but no rigorous justifications are provided. 

The paper is organized as follows:   In the next section, we start with the formulation of the problem
as an initial-boundary transmission problem and reduce it to a nonlocal initial-boundary  problem.
In Section 3,   we give a brief review of the Lubich  approach  and introduce the appropriate classes of operators from \cite{LaSa:2009a} and state and prove the crucial result concerning the inversion of the Laplace transform of the classes of operators introduced in this section.  Section 4 deals with the variational formulation of the nonlocal initial- boundary problem in Section 2.  Theorems 4.1 and 4.2 are the main existence and uniqueness results. The last section, Section 5,  states the main results  in the time domain. 
 
\section{Formulation of the problem}
\subsection{An initial -boundary transmission problem}
We are concerned with a time-dependent direct scattering problem in fluid-structure interaction, which can be simply described as follows: an acoustic wave propagates in a
fluid domain of infinite extent in which a bounded elastic body is
immersed. The problem is to determine the scattered
pressure and velocity fields in the fluid domain as well as the
displacement field in the elastic body at any time.  Throughout the paper, let $\Omega$ be the bounded domain in $\mathbb{R}^3$ occupied by the elastic body with a Lipschitz boundary $\Gamma$ and let  $\Omega^c = \mathbb{R}^3  \setminus  \overline\Omega$  its exterior occupied by a compressible fluid.  In the elastic domain $\Omega$, the  elastic displacement $\mathbf{u}(x,t)$ is governed by the 
dynamic linear elastic equation:
\begin{equation}\label{eq:2.1}
\rho_e \frac{\partial^2\mathbf{u}} {\partial t^2} - \Delta^{*} \mathbf{u} = \mathbf{0},  \quad (x,t) \in 
\Omega \times (0,T),
\end{equation}
where $T$ is a given positive constant and 
where $\rho_e$ is the constant density of the elastic body, and $\Delta^*$ is the Lam\'e operator 
\begin{eqnarray*}
 \Delta^* \mathbf{u} &:=&  \mu \Delta \mathbf{u} + (\lambda + \mu)  \nabla~div \, \mathbf{u}\\
 &=& div \; \bm{\sigma}(\bf u).
 \end{eqnarray*}
 Here $\bm {\sigma}({\bf u})$ and $\bm \varepsilon({\bf u}) $ are  the stress  and strain tensors, 
 respectively,
 $$ \bm \sigma({\bf u}) =(\lambda\;div\; {\bf u}){\bf I} + 2 \mu \bm \varepsilon({\bf u}) \quad \mbox{and }\quad
 \quad \bm \varepsilon({\bf u}) = \frac{1}{2}( \nabla{\bf u} + (\nabla{\bf u})^\prime) .$$
We assume that the elastic body is homogeneous  and  isotropic  with $\mu$ and $\lambda$  the 
corresponding Lam\'e constants, which are required to satisfy the constraints: $ \mu \geq 0,$ and  $3\lambda + 2 \mu \geq 0.$ 

In the  fluid domain  $\Omega^c$,  we consider a barotropic flow of an inviscid and compressible  fluid.  Let {\bf v}={\bf v} (x,t) be the velocity field,  and $p=p(x,t)$ and  $\rho= \rho(x,t)$ be respectively the pressure and the density of the fluid. We assume that ${\bf v},
 p $ and  $\rho$  are small perturbations of the static state ${ \bf v_0} = {\bf 0}, 
p_0 = constant $ and $ \rho_0 = constant. $  Then the  governing equations
may be linearized to yield  the linearized Euler equation
\begin{eqnarray}\label{eq:2.2} 
\rho_0\frac{\partial {\bf v}}{\partial  t} +  \nabla\; p &=&{ \bf 0}, 
\end {eqnarray}
the linearized equation of continuity 
\begin{eqnarray}\label{eq:2.3}
\frac{\partial \rho}{\partial t} + \rho_0 \, div \, {\bf v} &=& 0,
\end{eqnarray}
and the linearizecd state equation
\begin{eqnarray}\label{eq:2.4}  
p = c^2 \rho
\end{eqnarray}
in $\Omega^c \times (0, T)$, where $c$ is the sound speed defined by  
$ c^2 = f '( \rho_0) $ and $f $ is a function depending on the nature of the fluid (see e.g., \cite{Ac:1990, Se:1958}).

For an irrotational flow, this  formulation can be simplified in terms of a velocity potential $\varphi = \varphi(x,t)$  such that 
$$
{\bf v} =  - \nabla\; \varphi, \quad\mbox{and}\quad 
p= {\rho_0} \frac{\partial \varphi}{\partial t}. 
$$ 
Then it follows from (\ref{eq:2.3}) and (\ref{eq:2.4}),  the velocity potential $\varphi$ satisfies the wave equation 
\begin{equation}\label{eq:2.5}
\frac{\partial^2 \varphi}{\partial  t^2} - c^2 \Delta \varphi = 0 \quad\mbox{in}\quad \Omega^c \times (0, T).
\end{equation}
The time-dependent scattering problem can be formulated as {\em an initial-boundary transmission problem} consisting of  the partial differential equation (\ref{eq:2.1}) for the elastic displacement field  $\bf u$ and (\ref{eq:2.5}) for velocity 
potential $\varphi$ together with the homogeneous initial conditions 
\begin{equation}\label{eq:2.6}
 {\bf u}(x,0) = \frac{\partial {\bf u}(x,0)}{\partial t} = {\bf 0},\;  x \in \Omega \quad \mbox{and}
 \quad \varphi (x,0) = \frac{\partial \varphi} {\partial t} (x,0) = 0,\; x\in \Omega^c
  \end{equation}
 and the transmission conditions  on $\Gamma \times (0, T]$
 \begin{eqnarray}
 \bm \sigma({\bf u})^- {\bf n} &= & - {\rho_0}\;( \frac{\partial \varphi}{\partial t} + 
  \frac{\partial \varphi^{inc}}{\partial t})^+{\bf n},  \label{eq:2.7}\\
 \frac{ \partial {\bf u}^-} {\partial t} \cdot {\bf n}& = &- (\frac{\partial  \varphi }{\partial n} +
  \frac{\partial  \varphi^{inc}}{\partial n})^+ , \label{eq:2.8}
 \end{eqnarray}
where ${\bf n}$ is the exterior unit normal for $\Omega$, and $\varphi^{inc} $ denotes the given incident field. 
Here and in the sequel, we adopt the notation that $q^{\mp}$ denotes the limit of the function $q$ 
on $\Gamma$ from inside and outside, respectively. 
%%%%%%%%%%%%
%%%%%%%%%%%
\subsection{Reduction to a nonlocal initial- boundary problem}
Motivated by  time-harmonic   fluid-structure interaction problems \cite{HKR:2000}, we intend to apply the coupling of boundary integral and field equation methods to the 
 transmission  problem defined by (\ref{eq:2.1}) and  (\ref{eq:2.5} ) together with 
 (\ref{eq:2.6}) --(\ref{eq:2.8}).  The main idea here is  to convert the problem to a nonlocal problem 
 in a bounded computational domain such as $\Omega$  by a reduction of the solution in the fluid domain to appropriate boundary integral equations on the interface boundary $\Gamma$. 
 For the solution of the wave equation (\ref{eq:2.5}), we begin with the Kirchhoff formula (see e.g., \cite{HsWe:2011a, LaSa:2009a})
 \begin{equation}  \label{eq:REP}
 \varphi  = \mathcal{D} * \phi - \mathcal{S} *\lambda \quad  \mbox{in } \Omega^c \times (0, T), 
\end{equation}
where $\phi:=\varphi^+$ and $\lambda:=\partial \varphi^+/ \partial n$ are the Cauchy data of $\varphi$ on $\Gamma$ and
\begin{eqnarray}
(\mathcal{S} * \lambda)(x,t) &:=& \int_\Gamma E(x,y)  \lambda(y,t-\frac{|x-y|}{c} d\Gamma_y  \label{eq:2.10}\\
(\mathcal{D}* \phi)(x,t)  &:=& -\int_\Gamma \nabla_x \Big(E(x,y) \phi (y,t-\frac{|x-y|}{c}\Big) \cdot {\bf n}_y d\Gamma_y,  \label{eq:2.11}\\
&=&   \int_{\Gamma} \Big\{\frac{\partial E(x,y)}{\partial n_y } \phi(y,t-\frac{|x-y|}{c})  -\frac{1}{c} E(x,y) 
 \frac{\partial |x-y|}{\partial n_y}\phi(y,t-\frac{|x-y|}{c})\Big\}d\Gamma_y, \nonumber\\ 
 \label{eq:2.12}
\end{eqnarray}
are the retarded simple and double layer potentials, written in terms of the fundamental solution of the three dimensional Laplacian $E(x,y)=1/(4\pi |x-y|)$.
One may show  as in classical potential theory \cite{HsWe:2008} that  the Cauchy data $\phi$ and $\lambda$ at  smooth points  of $\Gamma$ are related   by the system of boundary integral equations form the form (see, e.g., \cite{ BaHa:1986a, BaHa:1986b, Co:2004,  Ha:1989})
\begin{eqnarray} 
\phi  &=&\Big( \frac{1}{2}\phi  + \mathcal{K} * \phi  \Big)  - \mathcal{V}* \lambda 
 \quad \mbox{on } \Gamma\times (0,T]  \label{eq:2.13},\\
\lambda  &=&- \mathcal{W} * \phi + \Big(\frac{1}{2} \lambda  - \mathcal{K}^{\prime} *\lambda  \Big)  \quad \mbox{on }\Gamma \times (0,T]\label{eq:2.14}.
\end{eqnarray}
The four retarded integral operators in \eqref{eq:2.13} and \eqref{eq:2.14} are called (in the order they appear in the formulas) double layer, simple layer, hypersingular, and transpose double layer operators. For instance, the explicit formulas for the operators in \eqref{eq:2.13} is
\begin{eqnarray}
(\mathcal{V} *\lambda)(x,t) &:=& \int_{\Gamma} E(x,y) \lambda (y,t-\frac{|x-y|}{c}) d\Gamma_y,  \label{eq:2.15}  \\
(\mathcal{K} *\phi)(x,t)  &:=&  \int_{\Gamma} \frac{\partial E(x,y)}{\partial n_y} \phi(y,t-\frac{|x-y|}{c}) d\Gamma_y -  
\frac{1}{c}(\mathcal{V}_r * \phi_t)(x,t)
\end{eqnarray}
where
\[
( \mathcal{V}_r *\psi)(x,t):= \int_{\Gamma}E(x,y) \frac{\partial}{\partial n_y}|x-y|
\psi(y, t -\frac{|x-y|}{c}) d\Gamma_y.
\]
The operator matrix 
defined by  the right-hand side of (\ref{eq:2.13}) and (\ref{eq:2.14}) resembles the familiar form of the Calder\'on  projector for the Laplacian in  potential theory (see e.g., \cite{HsWe:2008}). 

 In view of the transmission condition  (\ref{eq:2.8}),  we make a  substitution:
$$
\lambda = - \big( \frac{ \partial {\bf u}^-} {\partial t} \cdot {\bf n}   
+ \frac{\partial  \varphi^{inc +}}{\partial n} \big)
$$
in (\ref{eq:2.13}). This leads to the following nonlocal boundary problem  reads : Given $\varphi^{inc}$, find ${\bf u}$ in $\Omega \times (0,T] $ and 
$ \phi$ on $\Gamma \times (-\infty,T]$ satisfying  the following equations and conditions :
\begin{gather}
 \rho_e \frac{\partial^2\mathbf{u}} {\partial t^2} - \Delta^{*} \mathbf{u} = \mathbf{0},  \quad \mbox{in } 
\Omega \times (0,T),  \nonumber\\
 \bm \sigma({\bf u})^- {\bf n} =  - {\rho_0}\;( \phi_t + 
  \frac{\partial \varphi^{inc\;+}}{\partial t}){\bf n}\quad \mbox{on} \quad \Gamma \times (0,T], 
  \nonumber \\
 {\bf u}(x,0) = \frac{\partial {\bf u}(x,0)}{\partial t} = {\bf 0},\;  x \in \Omega,
\nonumber \\
 \quad \phi (x,t) =  0,\; x\in \Gamma, t \le 0, \nonumber \\
- \frac{1}{2}\phi + \mathcal K*\phi  +  \mathcal{V} * (\frac{ \partial {\bf u}^-} {\partial t}\cdot {\bf n_y})
=  - \mathcal{V} * \frac{\partial  \varphi^{inc +}}{\partial n}  \quad \mbox{on } \Gamma\times (0,T]. 
\label{eq:2.17}
\end{gather}
Note that the initial condition for $\phi$ has to be stated for negative values of $t$, given the fact that delays appear in the definition of the retarded integral operators. 
We may also replace (\ref{eq:2.17}) by (\ref{eq:2.13}) and (\ref{eq:2.14}) in the form:
\begin{gather} 
\Big( \frac{1}{2}\phi  - \mathcal{K}* \phi  \Big) + \mathcal{V} * \lambda  =0 
 \quad \mbox{on } \Gamma\times (0,T]  \label{eq:2.18},\\
-  \frac{ \partial {\bf u}^-} {\partial t} \cdot {\bf n}  + \mathcal{W} *\phi - \Big(\frac{1}{2} \lambda  - \mathcal{K}^{\prime}*\lambda \Big) 
= \frac{\partial  \varphi^{inc +}}{\partial n}  \quad \mbox{on } \Gamma \times (0,T]\label{eq:2.19}.
\end{gather}
In this case,  ${\bf u}$ in $\Omega \times (0,T)$,  
$ \phi$ and $\lambda$  on $\Gamma \times (-\infty,T]$ are the unknown for the solutions of the nonlocal boundary problem. In view of the definition of the boundary integral operators of 
$\mathcal{K}^\prime$ and $\mathcal{W}$, in addition to the homogeneous initial condition for $\phi$, we require that $ {\phi_t} (x,t) = 0$ and $\lambda(x,t) = 0,  \; x\in \Gamma, t\le 0$.

We notice that in the above formulations,   equations (\ref{eq:2.17}), (\ref{eq:2.18})
 and (\ref{eq:2.19}) are all nonlocal differential boundary integral equations. They are not only nonlocal in space but also nonlocal in time.  As pointed out in \cite{Co:2004},  it is not clear how to choose appropriate function spaces because of the retarded argument.  On the other hand,  it is known that  for the long time behavior of the solution, one may replace the nonlocal differential boundary  integral equation by an  appropriate approximated transparent  boundary condition (see, e.g. \cite{HsWe:2011a}). However, in general we prefer to employ the approach originally introduced by  Lubich in his study of convolution quadrature techniques for hyperbolic problems \cite{Lu:1994} (see also \cite{LuSc:1992} in the parabolic case). This approach has been  recently extended  systematically to  treating retarded potentials by Laliena and Sayas \cite{LaSa:2009a} by means  of   properties of  the operators in the frequency domain.  We remark  that  we will see  this technique  does not mean we are solving the problems in the transformed domain and then  applying  the inverse  Laplace transform to obtain the solutions in the time domain.  To illustrate the essence of this concept, in the next section  we begin with some preliminary results concerning the Laplace transforms of 
functions and operators with causality properties. 
\section{ Lubich's approach} %%
In this section, we give a brief review of the Lubich approach for treating time dependent boundary integral equations which has been advanced  by the work of  Laliena and Sayas. The presentation of this section follows their work in \cite{LaSa:2009b}. 

\subsection{The Laplace transform}
We consider the Laplace transform for causal distributions or operator-valued functions. Throughout the paper let the complex  plane be denoted  by $\mathbb{C}$ and its positive half-plane denoted by 
$$\mathbb{C}_+:= \{ s \in \mathbb{C} : Re~s > 0\}. $$
We begin with the Laplace transform for an ordinary complex-valued function. Let $f \colon [0, \infty) \to \mathbb{C}$ be a complex-valued function with limited growth at infinity.  The  Laplace transform of $f$ is defined by 
$$F(s)= \mathcal{L}\{f\}(s) := \int_0^\infty e^{-st} f(t) dt. $$  
A common criterion for limited growth at infinity is that $f$ be of exponential  order, but this is much too restrictive for the kind of problems we are treating here.
As in \cite{DaLi:1992, LaSa:2009b} one can define the Laplace transform for the case of causal continuous  linear maps  $f \colon \mathcal{S}(\mathbb{R}) \to X $ with limited growth at infinity  which  concept is defined as fellows for  tempered distributions defined on  $\mathcal{S}(\mathbb{R})$ with values in a complex  Banach space $X$. Here causal is taken as in the sense of $\mathcal{S}' (\mathbb{R})$ that $\langle f, \phi\rangle$ is zero element of $X$ for every $\phi$ in $\mathcal{S} (\mathbb{R})$ with support in $[0, \infty)$. Indeed 
for fixed $a$ and $b$ such that $-\infty < b <a <0 $ and a $C^{\infty}(\mathbb{R})$  function $\alpha$ which vanishes identically for $t \leq b$ and is identically equal to   $1$  for $t \geq a $ the function $\phi^*$, 
$$ \phi^* (t) := \alpha(t) e^{-st}$$
is in $\mathcal{S}(\mathbb{R}) $. Thus we define the Laplace transform of the casual continuous linear map $f$ as 
$$ 
\mathcal{L}\{f\}(s) := \langle f, \phi^*\rangle.$$
Moreover, this definition is independent of the choice of $a$ and $b$. 

We also want to consider the Laplace transform  of  a causal operator-valued function
$f \colon [0, \infty) \to {L}(X,Y)$,  where  $X$ and $Y$ are two complex Hilbert spaces and 
  $L(X,Y)$ is  the space of bounded linear operators from $X$ to $Y$.    The Laplace transform of $f$  is  defined by Bochner's integral
$$F(s)= \mathcal{L}\{f\}(s):= \int_0^\infty e^{-st} f(t) dt, $$  
if the integral exists.   We assume that $F(s)$ exists for all $s\in \mathbb{C}_+$ and decays fast enough at infinity so that inversion formula 
$$ f(t) = \frac{1}{2 \pi i} \int_{\sigma - i \infty}^{\sigma + i \infty} e^{st} F(s) ds $$
holds for all $\sigma = Re~ s > 0$. Now let $g\colon [0, \infty)  \to  X$ and let 
$$ (f*g) (t) := \int_0^t f(\tau) g(t-\tau) d \tau$$ 
denote the convolution  $f * g \colon [0, \infty)  \to Y$.   For appropriate  $f$ and $g$, 
we see that formally 
\begin{eqnarray}  
(f*g) (t)& :=& \int_0^t f(\tau) g(t-\tau) d \tau\nonumber\\
 &=& \int_0^t \Big(\frac{1}{2\pi i} \int_{\sigma - i \infty}^{\sigma + i \infty} e^{s \tau} F(s) ds\Big)
 g(t - \tau) d \tau \nonumber \\
 &=& \frac{1}{2 \pi i} \int_{\sigma - i \infty}^{\sigma + i \infty} F(s) \Big( \int_0^t e^{s\tau} g( t- \tau) \Big) ds, \label{eq:3.1}
\end{eqnarray} 
provided the changing orders of integrations can be adjusted by Fubini's theorem.  In fact, as we will see,  the relation in (\ref{eq:3.1}) is the essential idea  behind  the convolution  quadrature method introduced by Lubich since the late 80's   and has been employed for  treating time-dependent boundary integral equations in the early 90's (see, e.g., \cite{LuSc:1992, Lu:1994}).  We note that in the relation 
(\ref{eq:3.1}), the property of the convolution integral  $f*g$ in the time domain  depends upon  g in the time domain but on $f$  only in the transformed domain. The latter is more accessible.   In  the following, we  shall summarize some of the results in \cite{LaSa:2009b}) concerning  the precise conditions for the class of operators and functions for which relation (\ref{eq:3.1}) holds.  We begin with classes of operators.

\subsection{Classes of operators $\mathcal{A}(\mu, X,Y)$ and $\mathcal{E}(\mu, \theta, X)$ } 
\begin{itemize} 
\item {$\mathcal{A}(\mu, X,Y)$ }: For a given $\mu \in \mathbb{R} $, the elements of the class 
 $\mathcal{A}(\mu, X,Y)$ are the analytic functions $F\colon \mathbb{C}_+ \to L(X,Y)$ for which 
 there exists a real number $\mu$ such that for all $\sigma > 0$ there is $C_0(\sigma)$ such that 
 $$ ||F(s)|| \leq C_0(\sigma) |s|^{\mu}, \;\; \forall s \; \; s.t \; \; Re  ~ s > \sigma. $$
\item{$\mathcal{E}(\mu, \theta, X)$}: For given $\mu \in \mathbb{R}$ and a function
  $\theta \colon \mathbb{C}_+  \to  \mathbb{R}$,  we write $F \in \mathcal{E}(\mu, \theta, X) $  when 
$F  \colon \mathbb{C}_+ \to L(X,X^\prime) $ is analytic (where $X^\prime$ is the dual of $X$), and there exists a non-decreasing function 
 $c \colon (0, \infty) \to (0, \infty) $ such that 
 $$ Re~ \Big( e^{i \theta(s)} <F(s) \psi, \overline\psi>\Big) \geq \frac{c(Re~s)}{|s|^\mu} ||\psi||^2, \;\;\forall \psi \in X, \;\;\forall s \in \mathbb{C}_+. $$
 \end{itemize} 
The following theorems and  proposition are stated in \cite{LaSa:2009b}. The detailed proof of the first two appears in \cite{Sa:2012}, while the last one is a simple consequence of the Lax-Milgram lemma.

\begin{theorem} let $F \in \mathcal{A}(\mu, X, Y) $ with $\mu < -1$. Then there exists a continuous function $f \colon \mathbb{R} \to L(X,Y)$ such that supp $f \subset [0, \infty)$ such that its Laplace transform, defined in $\mathbb{C}_+$, is $F$. If $\mu < -k -1 $ with $k$ positive integer, then 
$f \in C^k (\mathbb{R}, L(X, Y)) $.
\end{theorem} 
\begin{theorem}
Let  $ F \in \mathcal{A}(\mu, X, Y) $  with $\mu  \geq  1 .$ Take $k$ such that $1 + \mu < k \leq 2 + \mu$. Then there exits $g \in C(\mathbb{R}, L(X,Y))$ such that $supp \;g \subset [0, \infty)$ and $F$ is the 
Laplace transform of $ g^{(k)}$  in $ \mathbb{C}_+$, where the derivative is understood in the sense of distributions in $\mathbb{R}$.
\end{theorem}
\begin{proposition} If $ F\in \mathcal{E}(\mu, \theta, X)$, then $F^{-1} \in \mathcal{A} (\mu, X^\prime, X)$.
\end{proposition}
For the inversion formula, we introduce the class of $X$-space-valued functions 
\begin{itemize}
\item {$\mathcal{A}(\mu, X)$}: Let  $X$ be a complex  Banach space and $\mu \in \mathbb{R}$.  We write $F \in \mathcal{A} (\mu, X)$ when $F$ is an analytic function
$$ F\colon \mathbb{C}_+ \to X
$$ 
satisfying 
$$ ||F(s)||_X \leq C_F(Re~s) |s|^\mu, \quad \forall s \in \mathbb{C}_+,
$$ 
where $C_F\colon (0, \infty) \to (0, \infty)$ is a non-increasing function such that 
$$ C_F(\sigma) \leq \frac{C}{\sigma^m}, \quad \forall \sigma \in (0, 1]
$$
with $C$ independent of $\sigma$.
\end{itemize}
Since $ 1 \leq |s|/(Re~s),$ it is clear that $\mathcal{A}(\mu, X) \subset \mathcal{A}(\mu + \varepsilon, X)$ for all $\varepsilon >0$.
\subsection{The inversion formula}
Let $F \in \mathcal{A}(\mu, X)$ with $\mu <-1$. For any $\sigma > 0$, we define 
\begin{equation} \label{eq:3.2}
 f(t) =\mathcal{L}^{-1}\{F\}:= \frac{1}{2 \pi i} \int_{\sigma -i \infty}^{\sigma +i \infty} e^{st} F(s) ds. 
\end{equation} 
as the inverse of the Laplace transformed function $F(s)$.
We can see that $f$  is well defined, since 
\begin{eqnarray}
|| f(t) ||_{X} & =& || \frac{1}{2 \pi} \int_0^{\infty} e^{(\sigma + i\omega)t} F(\sigma + i \omega)
d \omega ||_X \nonumber \\
&\leq& \frac{1}{2 \pi}C_{F}(\sigma) e^{\sigma t}  \int_0^{\infty} \frac{2} {|\sigma + i \omega|^{-\mu}} d \omega
, \quad - \mu >1, \nonumber\\
&=& \frac{1}{2 \pi}C_{F}(\sigma) e^{\sigma t} \sigma^{1+\mu}\int_0^{\infty} \frac{\zeta^{-1/2}}{(1 + \zeta)^{-\mu/2}} d \zeta,  \quad 
\mbox{with}\quad \zeta= \omega^2/\sigma^2, \nonumber\\
&=&  \frac{1}{2 \pi}C_{F}(\sigma) e^{\sigma t} \sigma^{1 + \mu} B(\frac{1}{2}, \frac{-(1+\mu)}{2}), 
\label{eq:3.3}
\end{eqnarray}
where the Euler Beta function $B$ is defined by 
$$B(z_1,z_2):= \int_0^\infty t^{(z_1 -1)} /(1+t)^{(z_1 + z_2)}dt $$
 with $Re~z_1, Re~z_2 >0$.

As a consequence of ({\ref{eq:3.3}), we have 
\begin{proposition} If $F  \in \mathcal{A}(\mu, X)$ with $\mu <-1$, then $F$ is the Laplace transform of a continuous causal function $f\colon \mathbb{R} \to X$ with polynomial growth. 
\end{proposition}
Finally, we include here the most crucial result for  the Lubich approach related to causal time convolutions \cite{Lu:1994}.  For the benefit of the reader,  we give a brief sketch of the proof. A slight improvement of this result  can be found in \cite{DoSa:2013} and \cite{Sa:2012}: it eliminates $t^\varepsilon$ in the right-hand side of \eqref{eq:3.4} and substitutes $g^{(k)}$ by a linear differential operator of order $k$ and constant coefficients acting on $g$.
\begin{theorem}\label{th:3.5}
Let $A= \mathcal{L}\{a\} \in \mathcal{A} (\mu, X,Y) $  with $\mu \geq 0$. Let 
$$ k:=\lfloor \mu +2 \rfloor,  \quad 
 \varepsilon:= k - (\mu +1) \in (0, 1].$$
If $g \in C^{k-1}(\mathbb{R}, X)$ is causal and $||g^{(k)}||_X$  is integrable, then $a*g \in C(\mathbb{R}, Y) 
$ is causal and 
\begin{equation} \label{eq:3.4}
||a*g(t)||_Y \leq c_{\varepsilon} ~ t^{\varepsilon} ~C_A(t^{-1}) ~\int_0^t ||g^{(k)} (\tau)||_X d \tau.
\end{equation}
\end{theorem}
\begin{proof}
Let $\mathcal{L}\{g\} = G(s)$.  Since $ a*g = \mathcal{L}^{-1}( A(s)G(s))$, we see that
$\mu - k = -1 -\varepsilon < -1$ and 
\begin{eqnarray*}
||a*g(t)||_Y& =& ||\mathcal{L}^{-1}\{ (s^{-k}A(s) \} * g^{(k)}||_Y  \\
&=&|| \int_0^t  \mathcal{L}^{-1} \{ (s^{(-k)} A(s) \} (\tau) g^{(k)}(t- \tau) d\tau ||_Y\\
&\leq& \int_0^t  \Big( \frac {1} {2 \pi} \int_0^\infty\frac  {2 C_A} {| \sigma +i \omega |^{1+ \varepsilon} }           d\omega  \Big)  e^{\sigma \tau}  || g^{(k)}  (t -\tau)  ||_X  d  \tau \\
&=& \frac{e^{\sigma t}}{2 \pi} C_A(\sigma) \sigma^{-\varepsilon} B(\frac{1}{2},  \frac{1}{\varepsilon})
\int_0^t || g^{(k)}  (t -\tau)  ||_X  d  \tau \\
&=& \frac{e^{\sigma t}}{2 \pi} C_A(\sigma)  \sigma^{-\varepsilon}  B(\frac{1}{2},  \frac{1}{\varepsilon})
\int_0^t || g^{(k)} (\tau) ||_X  d  \tau, 
\end{eqnarray*}
 By taking $\sigma =t^{-1}$ and $c_{\varepsilon} =  1/2 \pi~  e~ B(1/2,  1/\varepsilon)$, 
 this gives the estimate
(\ref{eq:3.4}).
\end{proof}
\subsection{An example}\label{s:3.4}
In order to illustrate the applicability of Theorem \ref{th:3.5}, and the concepts introduced in the section, we end this section by considering  a specific example.  Let us consider the simple retarded 
boundary integral operator, namely $\mathcal{V}$ in (\ref{eq:2.15})
\begin{equation*}
  (\mathcal{V} * \lambda)(x,t) := \int_{\Gamma} E(x,y)  \lambda(y, t- \frac{|x-y|}{c})d\Gamma_y, 
 \quad (x,t) \in \Gamma\times (0,T].
\end{equation*}
Note that the convolutional notation will be fully justified with this approach. It is also customary to write (at least formally)
\begin{eqnarray} 
 (\mathcal{V}*\lambda) (x,t) &= &\int_0^t \int_{\Gamma} E(x,t; y,\tau)  \lambda(y, \tau)d\Gamma_y d\tau,\label{eq:3.5}
 \end{eqnarray}
where $E(x,t;y,\tau)$ is the fundamental solution of  the wave operator 
$
\Box_{x,t}\varphi:= -  \Delta \varphi + \frac{1}{c^2} \partial^2 \varphi  / \partial
t^2, $ namely, 
\begin{equation}\label{eq:3.6}
E(x,t; y, \tau):= \delta\Big((t - \tau) - \frac{|x-y|}{c})\Big) E(x,y),
\end{equation}
where $\delta$ is the Dirac delta. 
Hence, we have 
\begin{eqnarray*}
\mathcal{L}\{\mathcal{V} * \lambda \} &= &V(s) \Lambda(s) , \quad x \in \Gamma,\; s \in \mathbb C_+,  \; i.e., \\
( \mathcal{V} *\lambda)(x,t)&=&\Big( \mathcal{L}^{-1}\{V(s)\}* \lambda \Big)(x,t),  
\quad x \in \Gamma,\; s \in \mathbb C_+.
\end{eqnarray*}
It can be shown that $V(s) $ is  just the simple - layer boundary integral operator for
the Laplace transform  of the wave operator,  
\begin{equation}\label{eq:3.7}
\mathcal{L} \{ \Box_{x,t} \varphi\}: = -\Delta \Phi + \frac{s^2}{c^2} \Phi, 
 \quad \Phi=\mathcal{L}\{\varphi\},
\end{equation}
which is defined explicitly as 
$$
V(s) \psi:= \int_{\Gamma} E_{s/c}(x,y)  \psi(y) d\Gamma_y,  \quad x\in \Gamma
$$
and 
$$E_{s/c}(x,y)= E(x,y) exp\{- {s |x,y| /c }\}  = \frac{ 1 }  {4 \pi |x -y|} exp\{ - s |x,y| / c  \} $$ 
is the fundamental solution of the transformed  wave operator  in (\ref{eq:3.7}). 

We now summarize the properties of the operator $V(s)$ as follows. Note that in our notation the angled bracket is linear in both components and thus symmetry is not to be confused with self-adjointness.
\begin{itemize}
\item{Symmetry}:
$$\langle \chi, V(s) \psi\rangle = \langle \psi, V(s) \chi \rangle, \quad \forall~ \chi, \psi 
\in H^{-1/2}(\Gamma). $$\\
\item{Positivity}:
$$Re \Big(e^{i \theta} \langle\overline{\psi}, V(s) \psi\rangle\Big) =\frac{\sigma}{|s|} |||u_{\psi}|||^2_{|s|, \mathbb{R}^3\setminus \Gamma}, \quad \forall \psi \in H^{-1/2}(\Gamma),$$
where  $\theta = Arg \;s$, the principal argument of $~s, \; u_{\psi} :=S(s)\psi$ in $\mathbb{R}^3 \setminus \Gamma$, and  $S(s)$ is the simple-layer potential 
corresponding to $V(s)$ with norm defined by
$$ |||u_{\psi}|||^2_{|s|, \mathbb{R}^3\setminus \Gamma}:= \int_{\mathbb{R}^3 \setminus \Gamma}
\Big\{|\nabla u_{\psi} |^2
+ \frac{|s|^2}{c^2} |u_{\psi}|^2 \Big\} dx.$$ 
\item{Coercivity}: $$Re \Big(e^{i \theta} \langle\overline{\psi}, V(s) \psi\rangle\Big) \geq C \frac{\sigma \underline{\sigma}}{|s|^2}\; ||\psi||^2_{H^{-1/2}(\Gamma)}. $$
Here and in the sequel,  $C$ is a generic constant independent of $s$. 
\item{Bounds}:
$$
||V(s)||_{L(X,X')} \leq  C \frac{|s|}{\sigma \underline\sigma^2}\quad , \quad
||V^{-1}(s) ||_{L(X', X)}\leq  C \frac{|s|^2}{\sigma \underline\sigma}.
$$
where $\sigma = Re~ s, \underline{\sigma}:= min \{1, \sigma\},  X=H^{-1/2}(\Gamma)$, and 
$X^\prime = H^{1/2}(\Gamma).$
\end{itemize}
We shall return to these properties later.  Most proofs of them are readily  available (see, e.g. \cite{BaHa:1986a, Lu:1994, LaSa:2009b} and \cite{Sa:2012}). As consequences of these properties, we see that 
$$ V  \in \mathcal{E}(2, \theta,
H^{-1/2}(\Gamma)) \cap \mathcal{A}(1, H^{-1/2}(\Gamma), H^{1/2}(\Gamma)),\; \mbox{and} \; V^{-1} \in \mathcal{A} (2, H^{1/2}(\Gamma), H^{-1/2}(\Gamma)). $$
Moreover, from Theorem\;\ref{th:3.5},  we have  the estimate for the simple retarded boundary integral operator 
\begin{equation} \label{eq:3.8}
||\mathcal{V}*\lambda(t)||_{H^{1/2}(\Gamma)} \leq c_{\varepsilon} ~ t^{\varepsilon} ~C_V(t^{-1}) ~\int_0^t ||\lambda^{(k)} (\tau)||_{H^{-1/2}(\Gamma)} d \tau,
\end{equation}
where $ \mu =1, k=\lfloor 1+2 \rfloor=3,  \varepsilon  =3 - (1 +1)=1,\,$ and $ C_V(t^{-1}) =  C~  t\; max\{1, t^2\}$. 
%%%%%%%%
%%%%%%%%
\section{Variational solutions} %Solutions in  transformed domain
We now return to the initial-boundary transmission problem defined by the
partial  differential equations (\ref{eq:2.1} ), (\ref{eq:2.5}), the  initial conditions 
(\ref{eq:2.6}), and the transmission conditions (\ref{eq:2.7}) and (\ref{eq:2.8}). Our first step is to consider the problem in the Laplace transformed domain.

\subsection{Formulation in the transformed domain}
In the following we let $\mathbf{U}(s) :=  \mathbf{U}(x,s)= \mathcal{L}\{ {\bf u}(x,t)\},  \Phi (s)
:=\Phi(x,s) = \mathcal{L}\{\varphi(x,t)\}$. Then the initial-boundary transmission problem consisting of
(\ref{eq:2.1}), (\ref{eq:2.5}),( \ref{eq:2.6}), (\ref{eq:2.7}) and (\ref{eq:2.8}) in the Laplace transformed domain becomes  the following transmission boundary value problem: 
\begin{eqnarray} 
 - \Delta^{*} \mathbf{U}(s)
 +  \rho_e s^2 \mathbf{U}(s)  &=& \mathbf{0} \quad \mbox{in} \quad  
\Omega, \label{eq:4.1}\\
-\Delta \Phi(s) + \frac{s^2}{c^2} \Phi(s) &=& 0 \quad \mbox{in}\quad  \Omega^c \label{eq:4.2} \\
\bm{\sigma} (\mathbf{U} )^- \mathbf{n} = - \rho_0\; s \Big( \Phi(s) + \Phi^{inc}  \Big) ^+ \mathbf{n}, 
&\mbox{and}&
s \mathbf{U}^- \cdot \mathbf{n} = - \Big( \frac{ \partial \Phi}{\partial n} + \frac {\partial \Phi^{inc}}{\partial n}\Big)^+ \; \mbox{on}\; \; \Gamma  \label{eq:4.3}
\end{eqnarray}
for $s\in \mathbb C_+$.  We remark that \eqref{eq:4.1}--\eqref{eq:4.3}  is an exterior scattering  problem,  and normally  a radiation condition is needed  in order to guarantee  the uniqueness  of the solution of the problem. In the present case, the radiation condition is substituted by the assumption that $\Phi\in H^1(\Omega^c)$, which is a Laplace-transform version of the weak Huygens principle.
 
To derive the proper nonlocal boundary problem, as usual, we begin via Green's third identity with 
the representation of the solutions of (4.2) in the form: 
\begin{equation}\label{eq:4.4}
 \Phi = D(s) \hat{\phi} - S(s)\hat{\lambda} \quad \mbox{in} \quad \Omega^c,
 \end{equation}
where $\hat{\phi}:= \Phi^+(s)$ and $\hat{\lambda}:= \partial \Phi^+ /\partial n$ are the Cauchy data for the operator in (\ref{eq:4.2}) and $S(s)$ and $D(s)$ are the simple-layer and duble-layer potentials
\begin{eqnarray}
S(s) \hat{\lambda} (x) &:=& \int_\Gamma  E_{s/c}(x,y) \hat{\lambda}(y) d\Gamma_y, 
 \quad x \in \Omega^c,
 \label{eq:4.5}\\
D(s) \hat{\phi} (x)  &:=& \int_\Gamma \frac{\partial}{\partial n_y} E_{s/c}(x,y) \hat{\phi}(y)  d\Gamma_y, \quad x \in \Omega^c
\label{eq:4.6}.
\end{eqnarray}
Here 
$$E_{s/c}(x,y) =  \frac{exp\{- s |x-y|/c\}} {4 \pi |x-y|}$$
 is the fundamental solution of the  operator  in  (\ref{eq:4.2}). By  standard arguments in potential
 theory,  we have the relations for the the Cauchy data $\hat{\lambda}$ and $\hat{\phi} $:
 \begin{equation}\label{eq:4.7}
\begin{pmatrix}
\hat{\phi} \\[3mm]
\hat{\lambda}\\
\end{pmatrix}
= \left (
\begin{matrix}  % or pmatrix or bmatrix or Bmatrix or ...
      \frac{1}{2}I + K(s) & -V(s) \\[3mm]
       -W(s)  &  ( \frac{1}{2}I - K(s))'  \\
    \end{matrix}
    \right )\begin{pmatrix}
\hat{\phi} \\[3mm]
\hat{\lambda}\\
\end{pmatrix}
\quad \mbox{on} \quad \Gamma.
 \end{equation}
Here $V, K, K^\prime $ and $W$  are the four  basic boundary integral operators familiar from potential theory such that 
\begin{eqnarray*}
V(s) \hat \lambda (x) &:=& \int_\Gamma  E_{s/c}(x,y) \hat{\lambda}(y) d\Gamma_y, 
 \quad x \in \Gamma\\
K(s) \hat{\phi} (x)  &:=& \int_\Gamma \frac{\partial}{\partial n_y} E_{s/c}(x,y) \hat{\phi}(y)  d\Gamma_y, \quad x \in \Gamma\\
K^\prime(s) \hat \lambda (x) &:=& \int_\Gamma  \frac{\partial}{\partial n_x} E_{s/c}(x,y) \hat{\lambda}(y) d\Gamma_y, 
 \quad x \in \Gamma,\\
W(s) \hat{\phi} (x)  &:=&- \frac{\partial}{\partial n_x}  \int_\Gamma \frac{\partial}{\partial n_y} E_{s/c}(x,y) \hat{\phi}(y)  d\Gamma_y, \quad x \in \Gamma
\end{eqnarray*} 
 By using the second transmission condition in (\ref{eq:4.3}), we obtain  from the second  boundary integral equation in  (\ref{eq:4.7}),  
\begin{equation}\label{eq:4.8}
-s \mathbf{U}^-  \cdot \mathbf{n} - ( \frac{1}{2} I -  K)'  \hat{\lambda}  + W \hat{\phi} =
\Big(\frac{\partial \Phi^{inc }}{\partial n}\Big)^+ \;\quad \mbox{on}\quad \Gamma
\end{equation} 
while  the  second boundary integral equation in (\ref{eq:4.7}) is simply 
\begin{equation} \label{eq:4.9}
( \frac{1}{2}I - K) \hat{\phi} + V \hat{\lambda} = 0 \quad \mbox{on} \quad \Gamma.
\end{equation}

On the other hand, the weak solution of (\ref{eq:4.1}) in $\Omega$  leads to the operator equation in 
of the form 
\begin{equation}\label{eq:4.10} 
{\mathbf A}_{ \Omega} \mathbf{U}  - \gamma^{\prime}(\bm{\sigma}(\mathbf {U})^-\mathbf{n}) = \mathbf{0} 
\end{equation}
in the dual of $\mathbf{H}^1(\Omega)$.  Here   $\mathbf{A}_{\Omega}(s)  \colon {\mathbf H}^1( \Omega)  \to ({\mathbf H}^1(\Omega))'$ is defined by
\begin{eqnarray*}
\langle \mathbf{A}_{ \Omega}(s)\mathbf{U},\mathbf{V}\rangle: = a( \mathbf{U},
 \mathbf{V})_{s, \Omega}
=
\int_{\Omega} \Big(
\lambda (div \,\mathbf {U})(div\,  \mathbf{V} )
+ 2 \mu \bm \varepsilon(\mathbf{U}) : \varepsilon ( \mathbf{V})  + \rho_{e} \,s^2 \mathbf{U}\cdot  \mathbf{V} \Big)\,  dx  && \\
\quad \mathbf U, \mathbf V \in \mathbf H^1(\Omega), &&
\end{eqnarray*}
where $ \lambda, \mu$  the Lam\'e constants. Also in \eqref{eq:4.10} we use $\gamma^\prime$, defined as the adjoint of the trace operator. Then by using the  first transmission  condition  in (\ref{eq:4.3}),
we substitute $ - \rho_0\; s \Big( \Phi(s) + \Phi^{inc}  \Big) ^+ \mathbf{n}$
for  $\bm{\sigma} (\mathbf{U} )^- \mathbf{n} $  into (\ref{eq:4.10}) which leads to the equation 
\begin{equation} \label{eq:4.11}
\tilde{\mathbf A}_{  \Omega}(s) \mathbf{U} +s \gamma^{\prime} ( \hat \phi\mathbf n) = -s \gamma' ((\Phi^{inc})^+ \mathbf{n}),
\end{equation}
where we have replaced $\tilde {\mathbf A}_{ \Omega}:=\rho^{-1}_0 {\mathbf A}_{ \Omega}$.
Collecting (\ref{eq:4.11}), (\ref{eq:4.8}) and (\ref{eq:4.9}), we arrive the following nonlocal problem, which reads: Given data $(\hat d_1, \hat d_2, \hat d_3)  \in  X' $, find $(\mathbf{U}, \hat\phi, \hat\lambda) 
\in X$ such that 
\begin{equation}\label{eq:4.12}
\pmb{\mathscr{A}}   \begin{pmatrix} % or pmatrix or bmatrix or Bmatrix or ...
      \mathbf{U}   \\
      \hat\phi\\
      \hat\lambda \\
   \end{pmatrix}:=
\left (   \begin{matrix}
   \tilde{\mathbf A}_{ \Omega}&s  ({\mathbf n}^\top \gamma)^\prime &0\\
   - s (\mathbf{n}^\top\gamma) & W & - ( \frac{1}{2} I -  K)' \\
   0 & ( \frac{1}{2}I - K) & V\\
   \end{matrix} 
\right )
   \begin{pmatrix}
    \mathbf{U}   \\
      \hat\phi\\
      \hat\lambda \\
    \end{pmatrix}
= \begin{pmatrix}
\hat d_1\\
\hat d_2\\
\hat d_3\\
   \end{pmatrix}, 
\end{equation}
where $(\hat d_1, \hat d_2, \hat d_3) : =( -s \gamma' ((\Phi^{inc})^+ \mathbf{n}), \;(\partial \Phi^{inc }/\partial n)^+,  \;0)$. The product spaces
$$ X ={\mathbf H}^1( \Omega)  \times H^{1/2}(\Gamma) \times H^{-1/2}(\Gamma) \quad \mbox{and}\quad
X^{\prime} =  ({\mathbf H}^1(\Omega))'\times H^{-1/2}(\Gamma) \times H^{1/2}(\Gamma)$$
are reciprocally dual Hilbert spaces.
Our aim is to show that Equation (\ref{eq:4.12}) has a unique solution in $X$. We will do this in the next subsection. However, before we do so, we will first show that  
$\pmb{\mathscr{A} }$
 is invertible. 

Using Gaussian elimination (as in  \cite{LaSa:2009b}), a simple computation shows that the matrix $\pmb{\mathscr{A}} $ 
 of operators can be decomposed in the form:
\begin{eqnarray} 
\pmb{\mathscr{A}} &=& \left (
\begin{matrix}
I & 0 &0 \\
 0& I & -(\frac{1}{2} I -  K)' V^{-1}\\
 0 & 0& I\\
  \end{matrix}\right ) 
\left ( \begin{matrix}
 \tilde{\mathbf A}_{ \Omega}&s  ({\mathbf n}^\top \gamma)^\prime&0 \\
- s\, ({\mathbf n}^\top \gamma)& B& 0\\
 0 & 0& V\\
 \end{matrix}\right )
 \left ( \begin{matrix}
I&0&0 \\
0&I&0\\
0 &  V^{-1}(\frac{1}{2} I -  K)& I\\
 \end{matrix} \right )  \nonumber \\   
&=:& P^{\prime} C P^{-1} ,   
\label{eq:4.13}         
\end{eqnarray}
where  $B := W + (\frac{1}{2} I -  K)' V^{-1}(\frac{1}{2} I -  K)$ .  We note that the operator matrix $C$  is {\em strongly elliptic}  (\cite{HsWe:2008, Mi:1970})  in the sense that 
\begin{eqnarray}\label{eq:4.14}
 Re \Big\{ \langle \Theta C (v, \psi, \chi) , \overline{ (v, \psi, \chi) } \rangle \Big\}
 &\geq & c(Re \, s)\;  |s|^{-2} || (v, \psi, \chi)||^2_X 
 \end{eqnarray}
for all $(\mathbf{v}, \psi, \chi) \in X $, where  $ \Theta$ is the matrix defined by 
$$ \Theta:= \left (\begin{matrix}
e^{- i \theta}&0 &0\\
0& e^{-i \theta} &0\\
0&0& e^{i \theta}\\
\end{matrix} \right ).
$$
Since both $P$ and $P^{\prime}$ are invertible, it follows from (\ref{eq:4.14})  that 
$\pmb{\mathscr{A}}$  is invertible. 
As for   the proof of  (\ref{eq:4.14}), we only want to point out  that 
$$ Re \Big\{s e^{-i\theta}  (  \langle \gamma^{\prime}(\psi  \mathbf{n}),  \bar{\mathbf{v}} \rangle   - 
\langle  \gamma  \mathbf{v} , \bar{\psi} \mathbf {n} \rangle ) \Big\} = 0, $$
so that 
\begin{eqnarray*}
 Re \Big\{ \langle \Theta C (\mathbf{v}, \psi, \chi) , \overline{ (\mathbf{v}, \psi, \chi) } \rangle \Big\}
 &= & Re\Big\{ e^{-i \theta}\langle \tilde{\mathbf A}_{  \Omega} \mathbf{v}, \bar{\mathbf{v}}\rangle +
 e^{-i \theta} \langle B\phi, \bar{\phi}\rangle + e^{i \theta} \langle V \chi,  \overline \chi \rangle\Big\}.\\
 \end{eqnarray*}
 It is clear that  what remans to be done is  to show that the operators  $\tilde{\mathbf A}_{ \Omega}, B$ 
 and $V$  belong to the appropriate class $\mathcal{E}(\mu, \theta, X)$  as $V$  in the example 
 (see \S \ref{s:3.4}).
The details of  the proof will  be omitted  here.  However, in order to show that 
$\pmb{\mathscr{A}}^{-1}$
 belongs to the appropriate  class $ \mathcal{A}(\mu, X, X^{\prime}) $ so that we may apply Theorem  \ref{th:3.5} to $\pmb{\mathscr{A}}^{-1}$
 for  obtaining   desired results in the time domain, we follow \cite{Sa:2012} in considering the existence and unique results of  a problem equivalent to the nonlocal problem defined by (\ref{eq:4.12}). 

Suppose that $(\mathbf{U}, \hat\phi, \hat\lambda)\in X$ is a solution of (\ref{eq:4.12}).  Let 
\begin{equation}\label{eq:4.15}
u:= D \hat\phi -  S \hat\lambda  \quad \mbox{in} \quad \mathbb{R}^3 \setminus \Gamma.
\end{equation}
Then  $u\in H^1(\mathbb{R}^3 \setminus \Gamma)$  is the solution of the transmission problem:
\begin{equation}\label{eq:4.16}
 - \Delta u + \frac{s^2}{c^2} u = 0 \quad \mbox{in}\quad  \mathbb{R}^3 \setminus \Gamma
 \end{equation}
satisfying  the following jump relations across $\Gamma$,  
$$[\gamma u]:=  \gamma^+ u - \gamma u=\hat\phi \in H^{1/2}(\Gamma), \quad  [\partial_n u]  :=\partial_n^+ u - \partial_n u= \hat\lambda \in H^{-1/2}(\Gamma). $$

First, from (\ref{eq:4.12}) we see that 
\begin{gather} 
 \tilde{\mathbf A}_{  \Omega}(s)\mathbf{U}+s \gamma^{\prime}([\gamma u]\mathbf{n})= \hat d_1 \quad \mbox{in} \quad \Omega.\label{eq:4.17}\\
-s \gamma\mathbf{U}\cdot \mathbf{n} - \partial^{+}_ n u = \hat d_2 \quad\mbox{on} \quad \Gamma, \label{eq:4.18}\\
- \gamma u = \hat d_3 \quad \mbox{on} \quad \Gamma. \label{eq:4.19}
 \end{gather}
 Since $\hat d_3 =0 $, this means that $u$ is a solution of the homogeneous Dirichlet problem 
 for the partial differential equation (\ref{eq:4.16}) in $\Omega$. Hence by the uniqueness of the the solution, we obtain  $u \equiv 0$ in  $\Omega$. Consequently, we have 
 \begin{equation} \label{eq:4.20}
 [\gamma u] = \gamma^+ u = \hat \phi \quad \mbox{and} \quad [\partial_n u] = \partial_n^+ u = \hat \lambda.
 \end{equation}
 
Next, we  consider the variational formulation of the problem for equations  (\ref{eq:4.16}) and (\ref{eq:4.17})  together with the boundary condition: (\ref{eq:4.18}).   We will seek  a  solution 
\begin{equation*}
(\mathbf{U}, u) \in  \pmb{\mathscr{H}}  =  {\mathbf H}^1(\Omega) \times H^1(\Omega^c)
\end{equation*}
 with the corresponding  test functions  $(\mathbf{V}, v)$ in the same function space.
 To derive the variational equations, we should keep in mind that the variational formulation should be formulated not in terms of  the Cauchy data $\hat\phi$ and $\hat \lambda$ directly but only in directly through the jumps of $u$ as indicted. 
 
We begin with the first  Green formula for the equation (\ref{eq:4.16}).  Let  $(u,v) \in  H^1(\Omega^c) \times  H^1(\Omega^c)$. Then
\[
 - \langle \partial^+_nu, \overline{\gamma^+ v} \rangle = \int_{\Omega^c} \Big\{\nabla u \cdot \overline{\nabla v} + \frac{s^2}{c^2} u\overline{ v}\Big\} dx
 =  a_{s,\, \Omega^c} (u, v)=: \langle A_{\Omega^c}(s)\; u,\overline{ v}\rangle.
\]
From condition  (\ref{eq:4.18}), we obtain 
\begin{equation} \label{eq:4.21}
 \langle A_{\Omega^c }(s)\, u,\overline{ v}\rangle  = - \langle \partial^+_n u, \overline{\gamma^+ v} 
 \rangle  
= \langle \hat d_2, \overline{\gamma^+ v}\rangle  + \langle  s \gamma\mathbf{U}\cdot \mathbf{n}, \overline{\gamma^+ v} \rangle. %%
\end{equation} 
Together with the weak formulation of (\ref{eq:4.17}) we arrive at the following variational formulation:
 Find $(\mathbf{U}, u) \in  \pmb{\mathscr{H}} $    satisfying
 \begin{equation} \label{eq:4.22}
\langle \tilde{\mathbf A}_{ \Omega}(s)\mathbf{U},\overline{ \mathbf{V}}\rangle
 + \langle A_{ \Omega^c}(s)\,u,\overline{ v}\rangle  
+s\Big\{ \langle \gamma^{\prime}((\gamma^+ u)\mathbf{n}), \overline{V}\rangle - \langle   \gamma\mathbf{U}\cdot \mathbf{n}, \overline{\gamma^+ v} \rangle\Big\} =\langle \hat d_1, \overline{\mathbf{V}}\rangle + \langle \hat d_2, \overline{\gamma^+ v}\rangle 
\end{equation}
for all $(\mathbf{V}, v) \in  \pmb{\mathscr{H}}$.
 We remark that by the construction, it can be shown that  as in \cite{Sa:2012} this variational problem is equivalent  to   the transmission problem defined by (\ref{eq:4.17}), (\ref{eq:4.16}),
 and  (\ref{eq:4.18}). The later is equivalent to the nonlocal problem defined by (\ref{eq:4.12}).  Consequently, the variational problem (\ref{eq:4.22}) is equivalent to the nonlocal problem (\ref{eq:4.12}).  Hence
for the existence of the solution of  (\ref{eq:4.12})},  it is sufficient to show the existence of the solution of (\ref{eq:4.22}).
\subsection{Existence and uniqueness results }
We recall that  $|||u|||_{|s|, \Omega^c}$ 
is  the  norm of  $u$ defined by
\begin{equation}
 |||u|||_{|s|, \Omega^c} :=  \Big\{ ||\nabla u ||^2_{0, \Omega^c}
+ \frac{|s|^2}{c^2} ||u||^2_{0, \Omega^c}\Big\}^{1/2} =\langle A_{ \Omega^c}(|s|)\, u,\overline{ u}\rangle^{1/2}. \label{eq:4.23}
\end{equation}
 Similarly, we define the norm of $\mathbf{U}$
 \begin{eqnarray}
 |||\mathbf{U}|||_{|s|, \Omega} &:=& \langle \mathbf{A}_{  \Omega}(|s|)\mathbf{U},\overline{\mathbf{U}}\rangle^{1/2}. \label{eq:4.24}
 \end{eqnarray}
 We also need the following  inequalities for the equivalent norms 
 \begin{gather}
\underline{\sigma } ||| \mathbf{U}|||_{1, \Omega} \leq  |||\mathbf{U}|||_{|s|, \Omega}\leq
\frac{|s|}{\underline{\sigma}}  |||\mathbf{U}|||_{1, \Omega} ,\label{eq:4.25}\\
\underline{\sigma} ||| u|||_{1, \Omega^c} \leq  |||u|||_{|s|, \Omega^c}  \leq 
\frac{|s|}{\underline {\sigma} }  |||u|||_{1,\Omega},   \label{eq:4.26}
 \end{gather}
which can be obtained  from the  inequalities: 
 $$\min\{1, \sigma\} \leq \min\{1, |s|\},\quad \mbox{and} \quad  \max\{1, |s|\} \min\{1, \sigma\}
 \leq |s|, \;\forall  s \in \mathbb{C}_+.$$
 We remark that the norm $ |||{ u}|||_{1, \Omega^c}$ is equivalent to $ ||{ u}||_{H^1(\Omega^c)}$ and  so is the energy norm $|||{ \mathbf{U}}|||_{1, \Omega}$  equivalent to the $\mathbf{H}^1(\Omega)$-norm of ${ \mathbf{U}} $ by the second Korn  inequality \cite{Fi:1972}.  In the following,  the ${c_j}'s $ are generic constants independent of $s$ which may or not not be the same at different places.  

 Then we have the following basic results.
\begin{theorem}\label{th:4.1}
The variational problem \em{ (\ref{eq:4.22})} has a unique solution $(\mathbf{U}, u) \in
\pmb{\mathscr{H}} $. Moreover,  the following estimates hold:
 \begin{eqnarray} 
 \Big\{ |||\mathbf{U}|||^2_{|s|, \Omega} + |||u|||^2_{|s|, \Omega^c}\Big\}^{1/2}
&\leq& c(\sigma, \underline{\sigma}) |s|  ||(\hat d_1, \hat d_2,  0)||_{X^{\prime}}, \label{eq:4.27}
\end{eqnarray}
where  $c(\sigma, \underline\sigma)$ is a constant depending only on  $\sigma$= {Re\, s} and $\underline{\sigma} = min\{1, \sigma\}$.
\end{theorem} 
 \begin{proof} 
 The existence and uniqueness results follow immediately from the identity 
 \begin{eqnarray}
Re \Big\{e^{-i\theta}\Big( \langle \tilde{\mathbf A}_{ \Omega}(s)\mathbf{U},\overline{ \mathbf{U}}\rangle & + & \langle A_{ \Omega^c}(s)\,u,\overline{ u}\rangle 
 + s\{ \langle \gamma^{\prime}((\gamma^+ u)\mathbf{n}), \overline{\mathbf U}\rangle - \langle  
 \gamma\mathbf{U}\cdot \mathbf{n}, \overline{\gamma^+ u} \rangle\} \Big)\Big\}  \nonumber\\ 
&=& Re \Big\{e^{-i\theta}\Big( \langle \tilde{\mathbf A}_{  \Omega}(s)\mathbf{U},\overline{ \mathbf{U}}\rangle  +  \langle A_{ \Omega^c}(s)\,u,\overline{ u}\rangle  \Big)\Big\}   \nonumber \\
&=&\frac{\sigma}{|s|} \Big(|||\mathbf{U} |||^2_{|s|, \Omega} +|||u|||^2_{|s|, \Omega^c}\Big). \label{eq:4.28}
\end{eqnarray}
For the estimate (\ref{eq:4.27}),  it follows from  (\ref{eq:4.28}) and (\ref{eq:4.22}) that
\begin{eqnarray*}
\frac{\sigma}{|s|} \Big(|||\mathbf{U} |||^2_{|s|, \Omega} +|||u|||^2_{|s|, \Omega^c}\Big)
&\leq& \Big|\langle \hat d_1, \overline{{\mathbf{U}}}\rangle + 
\langle \hat d_2, \overline{\gamma^+ {u}}\rangle\Big|\\
 &\leq&c_1 \Big\{ ||\hat d_1||_{({H^1(\Omega)})^\prime} |||{ \mathbf{U}}|||_{1, \Omega} 
 + ||\hat d_2||_{H^{-1/2}( \Gamma)} |||{ u}|||_{1, \Omega^c}\Big\}\\
&\leq& c_1\Big\{ |||{ \mathbf{U}}|||^2_{1, \Omega} + |||{ u}|||^2_{1, \Omega^c} \Big\}^{1/2} 
 \Big\{ ||\hat d_1||^2_{(H^1(\Omega))^\prime} + ||\hat d_2||^2_{H^{-1/2}(\Gamma)}\Big\}^{1/2}. 
 \end{eqnarray*}
 Consequently, we have the estimate 
 \begin{equation}\label{eq:4.29}
\Big\{|||\mathbf{U}|||^2_{|s|, \Omega} + |||u|||^2_{|s|, \Omega^c} \Big\}^{1/2}
\leq c(\sigma, \underline{\sigma})\; |s| \;||(\hat d_1,\hat d_2,0)||_{X^\prime},
\end{equation}
where $c(\sigma, \underline\sigma)= c_0/{\sigma \underline\sigma}$ with a constant $c_0$
independent of $s$ and $\sigma$.  In deriving the estimate (\ref{eq:4.29}), we 
 have tacitly applied the relations (\ref{eq:4.25}) and (\ref{eq:4.26}).
\end{proof}

As we will see  the  estimate (\ref{eq:4.29}) will lead us to show that the inverse of the operator
$\pmb{\mathscr{A}} $ in (\ref{eq:4.12}) belongs to the appropriate class  $\mathcal{A}(\mu, X,Y)$. 
In fact,  the following theorem holds for the operator $\pmb{\mathscr{A}}$ of (\ref{eq:4.13}). 
%%%%%%
%%%%%%
\begin{theorem} \label{th:4.2}
Let $ X =\mathbf{ H}^1(\Omega) \times H^{1/2}(\Gamma) \times H^{-1/2}(\Gamma) $ with its dual $ X^\prime =(\mathbf{ H}^1(\Omega))^\prime  \times H^{-1/2}(\Gamma) \times H^{1/2}(\Gamma) $
and with  $X^\prime_0 := \{ (\hat d_1, \hat d_2 ,\hat d_3) \in X^\prime\,  \big | \,  \hat  d_3 = 0\}$.  Then we have 
$$ \pmb{\mathscr{A}}  \in \mathcal{A}(2, X,X^\prime) \quad \mbox{and} \quad 
 \pmb{\mathscr{A}}^{-1}|_{X_0'} \in  \mathcal{A}(3/2 , X_0^\prime, X) .$$
 \end{theorem}
\begin{proof}
We will only prove $ \pmb{\mathscr{A}}^{-1}|_{X_0'} \in  \mathcal{A}(3/2 , X_0^\prime, X)$.  We remark that in the proof,  for simplicity, we will replace  the norm  $|| \cdot ||_{H^1(\Omega)} $ of $\mathbf{ H}^1(\Omega)$   by its equivalent norm  $ ||| \cdot|||_{1, \Omega}$.
It was pointed out in  (\ref{eq:4.20}) that 
$$
\gamma^+u=[\gamma u] =\hat\phi \in H^{1/2}(\Gamma), \quad \partial_n^+ u= [\partial_n u] =\hat\lambda \in H^{-1/2}(\Gamma). 
 $$
 Then we see that
\begin{equation}\label{eq:4.30}
||\hat\phi ||^2_{H^{1/2}(\Gamma)} = ||\gamma^+ u||^2_{H^{1/2}(\Gamma)}
\leq c_1 |||u|||^2_{1, \Omega^c}
\leq c_1 \frac{1}{\underline\sigma^2} |||u|||^2_{|s|,  \Omega^c}
\end{equation}
Similarly, we have 
$$
||\hat\lambda||^2_{H^{-1/2}(\Gamma)} =  || \partial^+_n u ||^2_{H^{-1/2}(\Gamma)} . $$
For $v \in H^{1/2}(\Gamma)$,  let $\tilde v $ be its extension as an $ H^1(\Omega^c)$ solution of 
$ - \Delta \tilde v + |s|^2/ c^2\, \tilde v = 0 \quad \mbox{in}\quad  \Omega^c$.
Then  by the generalized Green's first theorem (see , e.g., \cite{HsWe:2008} ) we have 
\begin{eqnarray*}
\langle \partial^+_n u, v\rangle &=&- \langle  A_{\Omega^c}(s)\; u, \tilde v \rangle \leq  |||u|||_{|s|, \Omega^c}  |||\tilde v|||_{|s|, \Omega^c}\\
&\leq& c_2 \;max\{ 1, |s|^{1/2} \} \;||v||_{H^{1/2}(\Gamma)} |||u|||_{|s|, \Omega^c}\\
&\leq& c_2 \Big( \frac{|s|}{\underline\sigma} \Big)^{1/2}\, ||v||_{H^{1/2}(\Gamma)} |||u|||_{|s|, \Omega^c},
\end{eqnarray*}
where on the right hand side, we have tacitly used the result in  \cite{BaHa:1986a}
(see also \cite{LaSa:2009a})   for  estimating $\tilde v$.
Hence, 
\begin{equation*} 
||\partial^+_n u ||_{H^{-1/2}(\Gamma)} : = sup  \frac{\langle \partial^+_n u, v\rangle }{  ||v||_{H^{1/2}(\Gamma) } } 
\leq c_2\Big( \frac{|s|} {\underline\sigma} \Big)^{1/2} |||u|||_{|s|, \Omega^c}.
\end{equation*}
Thus, we have the estimates
\begin{equation} \label{eq:4.31}
||\hat\lambda||^2_{H^{-1/2}(\Gamma)} \leq  c_2^2 \frac{|s|} {\underline\sigma} 
 |||u|||^2_{|s|, \Omega^c}.
\end{equation}
From (\ref{eq:4.30}) and (\ref{eq:4.31}),   we obtain the estimates
\begin{equation}\label{eq:4.32}
\frac{1}{2} \Big\{ c_1 \underline{\sigma}^2  ||\hat\phi ||^2_{H^{1/2}(\Gamma} + 
 \frac {\underline\sigma}{c_2^2|s|}  |\hat\lambda||^2_{H^{-1/2}(\Gamma)}\Big\}
 \leq 
 |||u|||^2_{|s|, \Omega^c}.
 \end{equation}
 As a consequence of (\ref{eq:4.29}),  it follows that 
 \begin{eqnarray*}
 \Big\{ \underline\sigma^2 ||| \mathbf{U}|||^2_{1, \Omega} + 
 \frac{1}{2} \Big(c_1\, \underline{\sigma}^2  ||\hat\phi ||^2_{H^{1/2}(\Gamma} + 
 \frac {\underline\sigma}{c_2^2|s|}  ||\hat\lambda||^2_{H^{-1/2}(\Gamma)}\Big)\Big\}&\leq& 
\Big\{  c(\sigma, \underline{\sigma}) |s|  ||(\hat d_1, \hat d_2,  0)||_{X^{\prime}}\Big\}^2.
 \end{eqnarray*}
However,  we see that 
\begin{eqnarray*}
LHS &\geq&\frac{1}{2} \frac{\underline\sigma}{|s|} \Big\{ \underline\sigma^2 ||| \mathbf{U}|||^2_{1, \Omega} + c_1
 \underline{\sigma}^2  ||\hat\phi ||^2_{H^{1/2}(\Gamma)} + \frac{1}{c_2^2} ||\hat\lambda||^2_{H^{-1/2}(\Gamma)}\Big\}\\
 &\geq&\frac{1}{2} \frac{\underline\sigma^3}{|s|} \Big\{ ||| \mathbf{U}|||^2_{1, \Omega} + c_1
  ||\hat\phi ||^2_{H^{1/2}(\Gamma)} + \frac{1}{c_2^2} ||\hat\lambda||^2_{H^{-1/2}(\Gamma)}\Big\}\\
\end{eqnarray*}
This implies that 
$$ \Big\{ ||| \mathbf{U}|||^2_{1, \Omega} + 
  ||\hat\phi ||^2_{H^{1/2}(\Gamma)} +  ||\hat\lambda||^2_{H^{-1/2}(\Gamma)}\Big\}^{1/2}
  \leq c_0(\sigma, \underline\sigma) |s|^{3/2} ||(\hat d_1,\hat d_2, 0)||_{X^\prime}, 
 $$
 where  
 $$c_0(\sigma, \underline\sigma)= \frac{C}{\sigma \underline\sigma^{5/2}}$$
 with constant $C$ independent of $s$ and $\sigma$.
Or, we have 
\begin{equation}\label{eq:4.33}
||(\mathbf{U}, \hat\phi, \hat\lambda)||_X \leq c_0(\sigma, \underline\sigma)\; |s|^{3/2} ||(\hat d_1,\hat d_2, 0)||_{X^\prime}
\end{equation}
which is  the desired result:   $\pmb{\mathscr{A}}^{-1}|_{X_0'} \in  \mathcal{A}(3/2 , X_0^\prime, X)$ .
\end{proof}

In view of (\ref{eq:4.4}), we see that $\mathbf{U}$ and $\Phi$ are solutions of the system 
 \begin{equation}\label{eq:4.34}
\begin{pmatrix}
\mathbf{U} \\[3mm]
\Phi\\
\end{pmatrix}
= \left (
\begin{matrix}  % or pmatrix or bmatrix or Bmatrix or ...
      I & 0 & 0\\[3mm]
      0 & D & -S  \\
    \end{matrix}
    \right ) 
\pmb{\mathscr{A}}^{-1} 
    \begin{pmatrix}
\hat{d_1} \\[3mm]
\hat{d_2}\\[3mm]
{0}\\
\end{pmatrix}.
 \end{equation}
As a consequence of  Theorems~\ref{th:4.1} and \ref{th:4.2} ,  we have the following corollary.
\begin{corollary}\label{co:4.3} The matrix of operators 
\begin{equation}\label{eq:4.35}
 \left (
\begin{matrix}  % or pmatrix or bmatrix or Bmatrix or ...
      I & 0 & 0\\[3mm]
      0 & D & -S  \\
    \end{matrix}
    \right )
 \pmb{\mathscr{A}}^{-1}|_{X_0'}  \quad \mbox{belongs to the class} \quad  \mathcal{A}(1 , X_0^\prime, \mathbf{H}^1(\Omega) \times H^1(\Omega^c)).
\end{equation}
\end{corollary}
\begin{proof}
We note that from (\ref{eq:4.27}), we have 
 \[
\Big\{ |||\mathbf{U}|||^2_{|s|, \Omega} + |||\Phi|||^2_{|s|, \Omega^c}\Big\}^{1/2}
= \Big\{ |||\mathbf{U}|||^2_{|s|, \Omega} + |||u|||^2_{|s|, \Omega^c}\Big\}^{1/2}
\leq c(\sigma, \underline{\sigma}) |s|\; ||(\hat d_1, \hat d_2,  0)||_{X^{\prime}}\;.
\]
\end{proof}

It should be mentioned  that for $\hat\phi \in H^{-1/2}(\Gamma)$, if  $ u= D(s) \hat\phi $ in $\mathbb{R}^3 \setminus \Gamma$, then 
\begin{eqnarray*} 
\frac{\sigma}{|s|} |||u|||^2_{|s|, \mathbb{R}^3\setminus \Gamma}&=& Re\; \Big\{ e^{-i \theta} \langle W\hat\phi, \overline{\hat\phi} \rangle \Big\}\\
&\leq& ||W\hat\phi||_{H^{-1/2}(\Gamma)} ||\hat\phi||_{H^{1/2}(\Gamma)} \\
&\leq &c_1\Big(\frac{|s|}{\underline\sigma}\Big)^{1/2} |||u|||_{|s|, \mathbb{R}^3 \setminus \Gamma}  ||\hat\phi||_{H^{1/2}(\Gamma)} .
\end{eqnarray*}
Hence from (\ref{eq:4.26}) we obtain that 
$$ ||D(s) \hat\phi||_{H^1( \mathbb{R}^3 \setminus \Gamma)} \leq  c_1 \frac{|s|^{3/2}}{\sigma {\underline\sigma}^{3/2}} ||\hat\phi||_{H^{1/2}(\Gamma)},$$
which implies $D \in 
 \mathcal{A}(3/2 , H^{/1/2}(\Gamma), H^1(\mathbb{R}^3 \setminus \Gamma)).
$
Similarly, we may show as in \cite{LaSa:2009b} that for $\hat \lambda$, if we set $ u= S(s) \hat\lambda $ in $\mathbb{R}^3 \setminus \Gamma$, then we can show that
$$||S(s) \hat\lambda||_{H^1(\mathbb{R}^3 \setminus \Gamma)} \leq c_2 \frac{|s|} {\sigma \underline\sigma^2}  ||\hat\lambda||_{H^{-1/2}(\Gamma)}, $$
and hence $S \in 
 \mathcal{A}(1 , H^{-1/2}(\Gamma), H^1(\mathbb{R}^3 \setminus \Gamma)).
$
This means 
\begin{equation*}
 \left (
\begin{matrix}  % or pmatrix or bmatrix or Bmatrix or ...
      I & 0 & 0\\[3mm]
      0 & D & -S  \\
    \end{matrix}
    \right ) \in  \mathcal{A} (3/2 , X,  \mathbf{H}^1(\Omega)  
\times H^1(\Omega^c)). 
\end{equation*} 
Following \cite{LaSa:2009b}, if we apply the composition rule and Theorem~\ref{th:4.2}, we find  the matrix of operators in (\ref{eq:4.34}) ended  with an index $\mu = 3/2 + 3/2 = 3$. However, this only gives an upper bound for the actual index as in the Corollary\;\ref{co:4.3}.
\section{Main results in the time domain } % in the time domain  
With the properties of the solutions in the transformed domains available, we now return to the solutions in the time domain.  As the example in \S \ref{s:3.4}  properties of the solutions in the time domains may be  obtained by applying  the inversion formula( \ref{eq:3.2}) to the  relevant   operators in the  transformed domain.  

We recall that  the matrix of operators  $\pmb{\mathscr{A}} $
 belongs to the class $\mathcal{A}(2, X,X^\prime)$ with  $ X =\mathbf{ H}^1(\Omega) \times H^{1/2}(\Gamma) \times H^{-1/2}(\Gamma) $ 
(see Theorem \ref{th:4.2}).  Indeed, we see that 
\begin{eqnarray}\label{eq:5.1}
||\pmb{\mathscr{A}}(s)||  &= & 3^2 \max_{1 \leq i,j  \leq 3}  ||\mathbb{A}_{ij}(s)||  \nonumber\\
&=& 3^2\; || W(s)|| 
\leq c_0 \frac{|s|^2}{\sigma \underline\sigma},
\end{eqnarray}
where $c_0$ is a constant  independent of $s$ and $\sigma$. 
We now apply  Theorem  \ref{th:3.5} to
$\pmb{\mathscr{A}}$, 
 with $\mu =2, k=4$  and $ \varepsilon =1$, and writing $\mathbf a:=\mathcal L^{-1}\{\pmb{\mathscr{A}}\}$.

%%%%%%%%
\begin{theorem}\label{th:5.1}
Let  $ X =\mathbf{ H}^1(\Omega) \times H^{1/2}(\Gamma) \times H^{-1/2}(\Gamma) $  and   
$\mathbf{g}(t) := \mathcal{L}^{-1} \{ ( {\mathbf U}, \hat \phi, \hat \lambda)^{\top} \}$. 
If $\mathbf{g} \in C^3([0, T], X)$ and $\mathbf{g}^{(4)}$  is integrable, then $
\mathbf a*\mathbf{g}$ belongs to $C([0, T], X^{\prime}) $  and 
\begin{equation} \label{eq:5.2}
||\mathbf a*\mathbf{g}(t)||_{X^{\prime}} \leq c_1 ~ t^2\, max\{1, t\} ~\int_0^t ||g^{(4)} (\tau)||_X\; d \tau.
\end{equation}
\end{theorem}
\noindent 
Similarly, by applying Theorem  \ref{th:3.5} to $\pmb{\mathscr{A}}^{-1}|_{X_0'}$
with $\mu =3/2,\; k=3$  and $ \varepsilon =1/2$, the following theorem holds.
\begin{theorem}\label{th:5.2}
Let  $ X =\mathbf{ H}^1(\Omega) \times H^{1/2}(\Gamma) \times H^{-1/2}(\Gamma) $  and $\mathbf{d}(t) := \mathcal{L}^{-1} \{ ( \hat{d}_1, \hat{d}_2, 0)^{\top} \}$. 
If $\mathbf{d} \in C^2([0, T], X^\prime)$ and $||\mathbf{d}^{(3)}||_{X^\prime}$  is integrable, then $
(\mathbf{u}, \phi, \lambda)^{\top}$ belongs to $C([0, T], X) 
$  and 
\begin{equation}  \label{eq:5.3}
||\begin{pmatrix}
\mathbf{u} \\
\phi\\
\lambda
\end{pmatrix}(t)||_X \leq c_{1/2} ~ t^{ \frac{1}{2} +1}\; max \{1, t^{2 \frac{1}{2} }\} ~\int_0^t ||\mathbf{d}^{(3)} (\tau)||_{X^{\prime}}\; d \tau.
\end{equation}
\end{theorem}

Finally in view of the Corollary \ref{co:4.3}, applying Theorem \ref{th:3.5} with $\mu =1, k = 3$ and $ \varepsilon =1 $,  the elastic and potential fields, $\mathbf{u}(x,t)$ and $\varphi(x,t)$ of the fluid-structure interaction  satisfy the estimates: 
\begin{theorem}\label{th:5.3}
Let  $ X =\mathbf{ H}^1(\Omega) \times H^{1/2}(\Gamma) \times H^{-1/2}(\Gamma) $  and    
$\mathbf{d}(t) := \mathcal{L}^{-1} \{ ( \hat{d}_1, \hat{d}_2, 0)^{\top} \}$. 
If $\mathbf{d} \in C^2([0, T], X^\prime)$ and $||\mathbf{d}^{(3)}||_{X\prime}$  is integrable, then $
(\mathbf{u}, \varphi)^{\top}$ belongs to $C([0, T], \mathbf{H}^1(\Omega) \times H^1(\Omega^c) ) 
$  and 
\begin{equation} \label{eq:5.4}
||\begin{pmatrix}
\mathbf{u} \\[3mm]
\varphi\\
\end{pmatrix}
(t)||_{\mathbf{H}^1(\Omega) \times H^1(\Omega^c)} \leq c_1 ~ t^2 \max \{1, t\} ~\int_0^t ||\mathbf{d}^{(3)} (\tau)||_{X^{\prime}}\; d \tau.
\end{equation}
\end{theorem}

We remark that Theorems \ref{th:5.1}-\ref{th:5.3} are mathematical foundations for the semi- and full-discretization  schemes based on the boundary element method and convolution quadrature method. We will pursue these investigations  in a separate communication. 
%\bibliography{HSW2013.bbl}

\end{document}